\documentclass{kluwer}    
\usepackage{amsfonts}
\usepackage{pstricks}
\usepackage{graphicx}

\newtheorem{theorem}{Theorem}

\newtheorem{definition}[theorem]{Definition}

\newtheorem{lemma}[theorem]{Lemma}
\newtheorem{proposition}[theorem]{Proposition}

\newenvironment{proof}[1][Proof]{\textbf{#1.} }{\ \rule{0.5em}{0.5em}}
\begin{document}
\begin{article}
\begin{opening}
\title{The Order Completion Method for Systems of Nonlinear PDEs Revisited}
\author{Jan Harm \surname{van der Walt}}
\runningauthor{J H van der Walt} \runningtitle{The Order
Completion Method For Systems Of Nonlinear PDEs}
\institute{Department of Mathematics and Applied Mathematics\\
University of Pretoria}
\date{}

\begin{abstract}
In this paper we presents further developments regarding the
enrichment of the basic Theory of Order Completion as presented in
\cite{Obergugenberger and Rosinger}.  In particular, spaces of
generalized functions are constructed that contain generalized
solutions to a large class of systems of continuous, nonlinear
PDEs. In terms of the existence and uniqueness results previously
obtained for such systems of equations \cite{vdWalt5}, one may
interpret the existence of generalized solutions presented here as
a regularity result.
\end{abstract}


\classification{Mathematics Subject Clasifications (2000)} {34A34,
54A20, 06B30, 46E05}

\end{opening}

\section{Introduction}

Consider a possibly nonlinear PDE, of order at most $m$, of the
form
\begin{eqnarray}
T\left(x,D\right)u\left(x\right)=f\left(x\right)\mbox{,
}x\in\Omega\subseteq \mathbb{R}^{n}\label{IntPDE}
\end{eqnarray}
with the righthand term $f$ a continuous function of $x\in
\Omega$, and the partial differential operator $T\left(x,D\right)$
defined by some jointly continuous mapping
\begin{eqnarray}
F:\Omega\times \mathbb{R}^{M}\rightarrow \mathbb{R}\nonumber
\end{eqnarray}
through
\begin{eqnarray}
T\left(x,D\right)u\left(x\right)=
F\left(x,u\left(x\right),...,D^{\alpha}
u\left(x\right),...\right)\mbox{, }|\alpha|\leq m
\label{IntPDEDef}
\end{eqnarray}
It is well know that an equation of the form (\ref{IntPDE})
through (\ref{IntPDEDef}) may, in general, fail to have a
classical solution $u\in\mathcal{C}^{m}\left(\Omega\right)$.
Moreover, there is in fact a \textit{physical interest} in
solutions to (\ref{IntPDE}) that are not classical.  From there
the interest in generalized solutions to nonlinear PDEs.

A long established idea in analysis is to obtain the existence of
generalized solutions to (\ref{IntPDE}) by associating with the
partial differential operator $T\left(x,D\right)$ a mapping
\begin{eqnarray}
T:X\ni u\mapsto Tu\in Y\label{IntPDEMap}
\end{eqnarray}
where $X$ is a relatively small space of classical functions on
$\Omega$, and $Y$ is some suitable space of functions with $f\in
Y$.  Appropriate topological structures, typically a norm or
locally convex topology, are defined on $X$ and $Y$ so that the
mapping $T$ is uniformly continuous with respect to these
customary structures. Generalized solutions to (\ref{IntPDE}) are
obtained by constructing the completions $X^{\sharp}$ and
$Y^{\sharp}$ of $X$ and $Y$, respectively, and extending the
mapping $T$ to a mapping
\begin{eqnarray}
T^{\sharp}:X^{\sharp}\rightarrow Y^{\sharp}\label{IntPDEMapExt}
\end{eqnarray}
A solution of the equation
\begin{eqnarray}
T^{\sharp}u^{\sharp}=f\label{IntPDEExt}
\end{eqnarray}
where the unknown $u^{\sharp}$ ranges over $X^{\sharp}$, is
considered a generalized solution of (\ref{IntPDE}).

As mentioned, the customary structures on the spaces $X$ and $Y$
in (\ref{IntPDEMap}) are typically locally convex linear space
topologies, or even normable topologies.  However, such methods,
involving the customary linear topological spaces of generalized
functions, appear ineffective in providing a \textit{general} and
\textit{type independent} theory for the existence and regularity
of solutions to nonlinear PDEs.  This apparent failure of the
usual methods of linear functional analysis in the study of
nonlinear PDEs is ascribed to the `complicated geometry of
$\mathbb{R}^{n}$' \cite{Arnold}. Moreover, in view of the above
mentioned inability of linear functional analysis and other
customary methods to yield such a general approach, it is widely
held that it is in fact \textit{impossible}, or at the very best
\textit{highly unlikely}, that such a theory exists.

This, as will be seen in the sequel, is in fact a
misunderstanding.  In this regard, we should mention that there
are currently two \textit{general} and \textit{type independent}
theories for the existence and regularity of generalized solutions
of nonlinear PDEs.  The Central Theory of PDEs \cite{Neuberger 1}
through \cite{Neuberger 4} is based on a generalized method of
steepest descent in suitably constructed Hilbert spaces.  This
method is fully type independent, that is, the particular form of
the operator that defines the equation is not used, and rather
general, but as of yet it is not universally applicable.   In
those cases where the method has been applied, it has resulted in
impressive numerical results. The Order Completion Method
\cite{Obergugenberger and Rosinger} and \cite{AnguelovRosinger},
on the other hand, constructs generalized solutions to a large
class of nonlinear PDEs in the Dedekind order completion of
suitable spaces of functions.  The essential feature of both
methods is that the spaces of generalized functions are tied to
the particular nonlinear partial differential operator
$T\left(x,D\right)$. Moreover, the underlying ideas upon which
they are based apply to situations that are far more general than
PDEs, this being exactly the reason for their respective type
independent power.

Recently, \cite{vdWalt5} the Order Completion Method was recast in
the setting of uniform convergence spaces \cite{Beattie}.  For a
system of $K$ nonlinear PDEs, each of order at most $m$, in $K$
unknown functions of the form
\begin{eqnarray}
\textbf{T}\left(x,D\right)\textbf{u}\left(x\right)=
\textbf{f}\left(x\right)\mbox{, }x\in\Omega
\subseteq\mathbb{R}^{n},\label{SystPDE}
\end{eqnarray}
where $\Omega$ is open, and $\textbf{f}$ is a continuous,
$K$-dimensional vector valued function on $\Omega$ with components
$f_{1},...,f_{K}:\Omega\rightarrow \mathbb{R}$, generalized
solutions are constructed as the elements of the completion of a
suitable uniform convergence space.  In particular, subject to a
mild assumption on the PDE (\ref{SystPDE}), namely
\begin{eqnarray}
\begin{array}{ll}
\forall & x\in\Omega\mbox{ :} \\
& \textbf{f}\left(x\right)\in \textrm{int}\{\textbf{F}\left(x,\xi\right)\mbox{ : }\xi\in\mathbb{R}^{M}\} \\
\end{array}\label{Assumption}
\end{eqnarray}
where $\textbf{F}:\Omega\times\mathbb{R}^{M}\rightarrow
\mathbb{R}^{K}$ is the jointly continuous function that defines
the system of PDEs (\ref{SystPDE}) through
\begin{eqnarray}
\textbf{T}\left(x,D\right)\textbf{u}\left(x\right)=
\textbf{F}\left(x,u_{1}\left(x\right),...,u_{K}\left(x\right)
,...,D^{\alpha}u_{i}\left(x\right),...\right)\mbox{, }|\alpha|\leq
m \mbox{ and } i=1,...,K, \label{SystPDEDef}
\end{eqnarray}
we obtain the existence and uniqueness of generalized solutions to
(\ref{SystPDE}).  Moreover, the generalized solution satisfies a
blanket regularity as it may be assimilated with nearly finite
normal lower semi-continuous functions.  In particular, there is
an uniformly continuous embedding from the space of generalized
solutions into the space of nearly finite normal lower
semi-continuous functions.

It should be noted that the assumption (\ref{Assumption}) is
hardly a restriction on the class of PDEs to which the method
applies.  Indeed, every linear PDEs, as well as most nonlinear
PDEs of applicative interest satisfy it trivially since, in these
cases,
\begin{eqnarray}
\{\textbf{F}\left(x,\xi\right)\mbox{ : }\xi\in \mathbb{R}^{M}\}
=\mathbb{R}^{K}.\nonumber
\end{eqnarray}
Therefore, the Order Completion Method \cite{Obergugenberger and
Rosinger} and the pseudo-topological version of the theory
\cite{vdWalt5} which we briefly discuss here, is to a large extent
universally applicable.

Nevertheless, one may notice that there remains a large scope for
possible enrichment of the basic theory. In particular, the space
of generalized solutions may depend on the nonlinear partial
differential operator (\ref{SystPDEDef}). Moreover, there is no
differential structure on the space of generalized functions
associated with the operator $\textbf{T}\left(x,D\right)$.  The
aim of this paper is to resolve these issues.  This is achieved by
setting up appropriate uniform convergence spaces \cite{Beattie},
somewhat in the spirit of Sobolev, which do not depend on the
particular operator $\textbf{T}\left(x,D\right)$.

The paper is organized as follows.  Section 2 introduces the
relevant spaces of functions upon which the appropriate spaces of
generalized functions are constructed.  In Section 3 we discuss
the approximation results underlying the Order Completion Method,
and introduce a suitable condition on the system of nonlinear PDEs
that allows the existence of generalized solutions in appropriate
spaces of generalized functions. These results are then applied in
Section 4 where we prove the existence of generalized solutions.
The structure of the spaces of generalized functions, and that of
the generalized functions that are their elements, is discussed in
Section 5.

\section{Function Spaces and Their Completions}

Recall \cite{vdWalt3}, \cite{sam}, \cite{Dilworth} that an
extended real valued function $u:\Omega\rightarrow
\overline{\mathbb{R}}$ is normal lower semi-continuous if
\begin{eqnarray}
\left(I\circ S\right)\left(u\right)\left(x\right)=
u\left(x\right)\mbox{, }x\in\Omega\label{NLSDef}
\end{eqnarray}
where
\begin{eqnarray}
I\left(u\right)\left(x\right)= \sup\{\inf\{u\left(y\right)\mbox{ :
}y\in\Omega\mbox{, }\|x-y\|<\delta\}\mbox{ :
}\delta>0\}\label{IDef}
\end{eqnarray}
and
\begin{eqnarray}
S\left(u\right)\left(x\right)= \inf\{\sup\{u\left(y\right)\mbox{ :
}y\in\Omega\mbox{, }\|x-y\|<\delta\}\mbox{ :
}\delta>0\}\label{SDef}
\end{eqnarray}
are the the Lower and Upper Baire Operators respectively, see
\cite{Anguelov} and \cite{Baire}.  These operators, as well as
their composition, are monotone with respect to the pointwise
ordering of functions $u:\Omega\rightarrow \overline{\mathbb{R}}$,
and idempotent. A normal lower semi-continuous function is said to
be nearly finite if
\begin{eqnarray}
\{x\in\Omega\mbox{ : }u\left(x\right)\in\mathbb{R}\}\mbox{ open
and dense}\nonumber
\end{eqnarray}
The space of all nearly finite normal lower semi-continuous on
$\Omega$ is denoted by $\mathcal{NL}\left(\Omega\right)$.  These
functions satisfy the following well know property of continuous
real valued functions, namely,
\begin{eqnarray}
\begin{array}{ll}
\forall & u,v\in\mathcal{NL}\left(\Omega\right)\mbox{ :} \\
\forall & D\subseteq \Omega\mbox{ dense :} \\
& u\left(x\right)\leq v\left(x\right)\mbox{, }x\in D\Rightarrow u\left(x\right)\leq v\left(x\right)\mbox{, }x\in \Omega \\
\end{array}\label{IneqDSet}
\end{eqnarray}
Moreover, for every $u\in\mathcal{NL}\left(\Omega\right)$ there is
a set $B\subset\Omega$ of first Baire Category such that
$u\in\mathcal{C}\left(\Omega\setminus B\right)$.

Clearly, each function that is continuous is also normal lower
semi-continuous. For $0\leq l\leq \infty$, the subspace of
$\mathcal{NL}\left(\Omega\right)$ consisting of functions that are
continuous with continuous partial derivatives up to order $l$ on
some open and sense subset of $\Omega$ is denoted
$\mathcal{ML}^{l}\left(\Omega\right)$.  That is,
\begin{eqnarray}
\mathcal{ML}^{l}\left(\Omega\right)=
\left\{u\in\mathcal{NL}\left(\Omega\right)\begin{array}{|ll}
\exists & \Gamma\subset\Omega\mbox{ closed nowhere dense :} \\
& u\in\mathcal{C}^{l}\left(\Omega\setminus\Gamma\right) \\
\end{array}\right\}\label{MLDef}
\end{eqnarray}
The space $\mathcal{NL}\left(\Omega\right)$, ordered in a
pointwise way, is a fully distributive lattice \cite{vdWalt3}, and
contains $\mathcal{ML}^{0}\left(\Omega\right)$ as a sublattice.
Therefore \cite{vdWalt1} the order convergence of sequences
\cite{RieszI} on $\mathcal{ML}^{0}\left(\Omega\right)$, which is
defined through
\begin{eqnarray}
\left(u_{n}\right)\mbox{ order converges to }u\Leftrightarrow
\left( \begin{array}{ll}
\exists & \left(\lambda_{n}\right)\mbox{, }\left(\mu_{n}\right)\subset \mathcal{ML}^{0}\left(\Omega\right)\mbox{ :} \\
& \begin{array}{ll}
1) & \lambda_{n}\leq \lambda_{n+1}\leq u_{n+1} \leq \mu_{n+1}\leq \mu_{n}\mbox{, }n\in\mathbb{N} \\
2) & \sup\{\lambda_{n}\mbox{ : }n\in\mathbb{N}\}=u=\inf\{\mu_{n}\mbox{ : }n\in\mathbb{N}\} \\
\end{array} \\
\end{array}\right)\label{OConv}
\end{eqnarray}
is induced by a convergence structure \cite{Beattie}.  In fact,
the uniform convergence structure $\mathcal{J}_{o}$ \cite{vdWalt3}
induces the order convergence of sequences, and is defined as
follows.
\begin{definition}\label{JoDef}
A filter $\mathcal{U}$ on
$\mathcal{ML}^{0}\left(\Omega\right)\times\mathcal{ML}^{0}\left(\Omega\right)$
belongs to the family $\mathcal{J}_{o}$ whenever there exists
$k\in\mathbb{N}$ such that, for every $i=1,...,k$, there exists
nonempty order intervals $\left(I_{n}^{i}\right)_{n\in\mathbb{N}}$
such that following conditions are satisfied:
\begin{eqnarray}\nonumber
&\mbox{1) }&\mbox{$I_{n}^{i}\supseteq I_{n+1}^{i}$ for every
$n\in\mathbb{N}$.}\nonumber\\
&\mbox{2) }&\mbox{If $V\subseteq\Omega$ is open, then
$\displaystyle\bigcap_{n\in\mathbb{N}}I^{i}_{n|V}=\emptyset$, or
there exists
$u_{i}\in\mathcal{ML}^{0}\left(V\right)$ such}\nonumber\\
&&\mbox{that
$\displaystyle\bigcap_{n\in\mathbb{N}}I^{i}_{n|V}=\{u_{i}\}$.}\nonumber\\
&\mbox{3) }& \left(\mathcal{I}^{1}\times
\mathcal{I}^{1}\right)\bigcap... \bigcap\left(\mathcal{I}^{k}
\times\mathcal{I}^{k}\right) \subset\mathcal{U}\nonumber
\end{eqnarray}
Here $\mathcal{I}^{i}=[\{I^{i}_{n}\mbox{ : }n\in\mathbb{N}\}]$ and
$I^{i}_{n|V}$ consists of all functions $u\in I^{i}_{n}$,
restricted to $V$.
\end{definition}
The uniform convergence structure $\mathcal{J}_{o}$ is uniformly
Hausdorff and first countable.  Moreover, a filter $\mathcal{F}$
on $\mathcal{ML}^{0}\left(\Omega\right)$ converges to
$u\in\mathcal{ML}^{0}\left(\Omega\right)$ with respect to
$\mathcal{J}_{o}$ if and only if
\begin{eqnarray}
\begin{array}{ll}
\exists & \left(\lambda_{n}\right)\mbox{, }\left(\mu_{n}\right)\subset\mathcal{ML}^{0}\left(\Omega\right)\mbox{ :} \\
& \begin{array}{ll}
1) & \lambda_{n}\leq \lambda_{n+1}\leq \mu_{n+1}\leq\mu_{n}\mbox{, }n\in\mathbb{N} \\
2) & \sup\{\lambda_{n}\mbox{ : }n\in\mathbb{N}\}=u=\inf\{\mu_{n}\mbox{ : }n\in\mathbb{N}\} \\
3) & [\{[\lambda_{n},\mu_{n}]\mbox{ : }n\in\mathbb{N}\}]\subseteq\mathcal{F} \\
\end{array} \\
\end{array}\label{OCStructure}
\end{eqnarray}
The completion of the uniform convergence space
$\mathcal{ML}^{0}\left(\Omega\right)$ is obtained as the space
$\mathcal{NL}\left(\Omega\right)$ equipped with an appropriate
uniform convergence structure $\mathcal{J}_{o}^{\sharp}$, see
\cite{vdWalt3}.  In particular, the uniform convergence structure
$\mathcal{J}_{o}^{\sharp}$ induces the order convergence structure
(\ref{OCStructure}).

The usual partial differential operators on
$\mathcal{C}^{l}\left(\Omega\right)$, with $l\geq 1$, may be
extended to $\mathcal{ML}^{l}\left(\Omega\right)$ through
\begin{eqnarray}
\mathcal{D}^{\alpha}:\mathcal{ML}^{l}\left(\Omega\right)\ni
u\mapsto \left(I\circ S\right)\left(D^{\alpha} u\right) \in
\mathcal{ML}^{0}\left(\Omega\right)\label{ExtPDO}
\end{eqnarray}
Therefore, somewhat in the spirit of Sobolev, we equip the space
$\mathcal{ML}^{l}\left(\Omega\right)$, where $l\geq 1$, with the
initial uniform convergence structure $\mathcal{J}_{D}$ with
respect to the family of mappings
\begin{eqnarray}
\left(\mathcal{D}^{\alpha}:\mathcal{ML}^{l}\left(\Omega\right)
\rightarrow \mathcal{ML}^{0}\left(\Omega\right)\right)
_{|\alpha|\leq l}\label{Fam}
\end{eqnarray}
That is, for any filter $\mathcal{U}$ on
$\mathcal{ML}^{l}\left(\Omega\right)\times
\mathcal{ML}^{l}\left(\Omega\right)$, we have
\begin{eqnarray}
\mathcal{U}\in\mathcal{J}_{D}\Leftrightarrow
\left(\begin{array}{ll}
\forall & |\alpha|\leq l \\
& \left(\mathcal{D}^{\alpha}\times \mathcal{D}^{\alpha}\right)\left(\mathcal{U}\right)\in\mathcal{J}_{o} \\
\end{array}\right)\label{JDDef}
\end{eqnarray}
Since the family of mappings (\ref{Fam}) separates the elements of
$\mathcal{ML}^{l}\left(\Omega\right)$, that is,
\begin{eqnarray}
\begin{array}{ll}
\forall & u,v\in\mathcal{ML}^{l}\left(\Omega\right)\mbox{ :} \\
\exists & |\alpha|\leq l\mbox{ :} \\
& \mathcal{D}^{\alpha}u\neq \mathcal{D}^{\alpha}v \\
\end{array},\nonumber
\end{eqnarray}
it follows that $\mathcal{J}_{D}$ is uniformly Hausdorff.  A
filter $\mathcal{F}$ on $\mathcal{ML}^{l}\left(\Omega\right)$ is a
Cauchy filter if and only if
$\mathcal{D}^{\alpha}\left(\mathcal{F}\right)$ is a Cauchy filter
in $\mathcal{ML}^{0}\left(\Omega\right)$ for each $|\alpha|\leq
l$.  In particular, a filter $\mathcal{F}$ on
$\mathcal{ML}^{l}\left(\Omega\right)$ converges to
$u\in\mathcal{ML}^{l}\left(\Omega\right)$ if and only if
$\mathcal{D}^{\alpha}\left(\mathcal{F}\right)$ converges to
$\mathcal{D}^{\alpha}u$ in $\mathcal{ML}^{0}\left(\Omega\right)$
for each $|\alpha|\leq l$.  In view of the results \cite{vdWalt2}
on the completion of uniform convergence spaces, the completion of
$\mathcal{ML}^{l}\left(\Omega\right)$ is realized as a subspace of
$\mathcal{NL}\left(\Omega\right)^{M}$, for an appropriate
$M\in\mathbb{N}$. In analogy with the case $l=0$ we denote the
completion of $\mathcal{ML}^{l}\left(\Omega\right)$ by
$\mathcal{NL}^{l}\left(\Omega\right)$.  The structure of the space
$\mathcal{NL}^{l}\left(\Omega\right)$ and its elements will be
discussed in Section 5.

With a system of PDEs of the form (\ref{SystPDE}) we may associate
a mapping
\begin{eqnarray}
\textbf{T}:\mathcal{ML}^{m}\left(\Omega\right)^{K}\rightarrow
\mathcal{ML}^{0}\left(\Omega\right)^{K}\label{SystPDEMap}
\end{eqnarray}
Indeed, we can write the system (\ref{SystPDE}) componentwise as
\begin{eqnarray}
\begin{array}{ccc}
T_{1}\left(x,D\right)\textbf{u}\left(x\right) & = & f_{1}\left(x\right) \\
\vdots & \vdots & \vdots \\
T_{j}\left(x,D\right)\textbf{u}\left(x\right) & = & f_{j}\left(x\right) \\
\vdots & \vdots & \vdots \\
T_{K}\left(x,D\right)\textbf{u}\left(x\right) & = & f_{K}\left(x\right) \\
\end{array}\label{SystPDEComp}
\end{eqnarray}
where, for each $j=1,...,K$ the component $T_{j}\left(x,D\right)$
of $\textbf{T}\left(x,D\right)$ is defined through the component
$F_{j}$ of the mapping $\textbf{F}$ by
\begin{eqnarray}
T_{j}\left(x,D\right)\textbf{u}\left(x\right) =
F_{j}\left(x,u_{1}\left(x\right),...,u_{K}\left(x\right),...,
D^{\alpha}u_{i}\left(x\right),...\right)\mbox{, }|\alpha|\leq
m\mbox{ and }i=1,...,K \label{SystPDECompDef}
\end{eqnarray}
The mappings (\ref{SystPDECompDef}) are then extended to
$\mathcal{ML}^{m}\left(\Omega\right)$ through
\begin{eqnarray}
T_{j}:\mathcal{ML}^{m}\left(\Omega\right)^{K}\ni\textbf{u}\mapsto
\left(I\circ S\right)\left(F_{j}\left(\cdot,u_{1},...,u_{K},...,
\mathcal{D}^{\alpha} u_{i},...\right)\right)\in
\mathcal{ML}^{0}\left(\Omega\right),\label{SystPDEMapComp}
\end{eqnarray}
yielding the components of the mapping (\ref{SystPDEMap}).  An
equivalence relation is induced on $\mathcal{ML}^{m}\left(
\Omega\right)^{K}$ by the mapping $\textbf{T}$ through
\begin{eqnarray}
\begin{array}{ll}
\forall & \textbf{u}\mbox{, }\textbf{v}\in\mathcal{ML}^{m}\left(\Omega\right)^{K}\mbox{ :} \\
& \textbf{u}\sim_{\textbf{T}}\textbf{v} \Leftrightarrow \textbf{Tu}=\textbf{Tv} \\
\end{array}\label{TEquiv}
\end{eqnarray}
The resulting quotient space
$\mathcal{ML}^{m}\left(\Omega\right)^{K}/\sim_{\textbf{T}}$ is
denoted $\mathcal{ML}^{m}_{\textbf{T}}\left(\Omega\right)$.  With
the mapping (\ref{SystPDEMap}) we may associate in a canonical way
an \textit{injective} mapping
\begin{eqnarray}
\widehat{\textbf{T}}:\mathcal{ML}^{m}_{\textbf{T}}\left(\Omega\right)
\rightarrow \mathcal{ML}^{0}\left(\Omega\right)^{K}\label{THat}
\end{eqnarray}
so that the diagram

\setlength{\unitlength}{1cm} \thicklines
\begin{picture}(13,6)
\put(2.0,4.7){$\mathcal{ML}^{m}\left(\Omega\right)^{K}$}
\put(3.8,4.8){\vector(1,0){6.5}}
\put(10.5,4.7){$\mathcal{ML}^{0}\left(\Omega\right)^{K}$}
\put(6.6,5.1){$\textbf{T}$} \put(3.0,4.5){\vector(0,-1){3.3}}
\put(2.3,2.7){$q_{\textbf{T}}$} \put(11.5,2.7){$id$}
\put(2.3,0.8){$\mathcal{ML}^{m}_{\textbf{T}}\left(\Omega\right)$}
\put(4.1,0.8){\vector(1,0){6.2}}
\put(6.6,1.0){$\widehat{\textbf{T}}$}
\put(10.5,0.7){$\mathcal{ML}^{0}\left(\Omega\right)^{K}$}
\put(11.2,4.5){\vector(0,-1){3.3}}
\end{picture}
\\
\\
commutes.  Here $q_{\textbf{T}}$ denotes the canonical quotient
map associated with the equivalence relation (\ref{TEquiv}), and
$id$ is the identity on $\mathcal{ML}^{0}\left(\Omega\right)^{K}$.

The product spaces $\mathcal{ML}^{0}\left(\Omega\right)^{K}$ and
$\mathcal{ML}^{m}\left(\Omega\right)^{K}$ will carry, naturally,
the product uniform convergence structures $\mathcal{J}_{o}^{K}$
and $\mathcal{J}_{D}^{K}$, respectively. That is,
\begin{eqnarray}
\mathcal{U}\in\mathcal{J}_{o}^{K}\Leftrightarrow
\left(\begin{array}{ll}
\forall & i=1,...,K\mbox{ :} \\
& \left(\pi_{i}\times\pi_{i}\right)\left(\mathcal{U}\right)\in\mathcal{J}_{o} \\
\end{array}\right)\label{ML0ProdUCS}
\end{eqnarray}
and
\begin{eqnarray}
\mathcal{U}\in\mathcal{J}_{D}^{K}\Leftrightarrow
\left(\begin{array}{ll}
\forall & i=1,...,K\mbox{ :} \\
& \left(\pi_{i}\times\pi_{i}\right)\left(\mathcal{U}\right)\in\mathcal{J}_{D} \\
\end{array}\right)\label{ML0ProdUCS}
\end{eqnarray}
Here $\pi_{i}$ denotes the projection on the $i$th coordinate. The
completion of $\mathcal{ML}^{0}\left(\Omega\right)^{K}$ is
$\mathcal{NL}\left(\Omega\right)^{K}$ equipped with the product
uniform convergence structure induced by
$\mathcal{J}_{o}^{\sharp}$.  Similarly, the completion of
$\mathcal{ML}^{m}\left(\Omega\right)^{K}$ is
$\mathcal{NL}^{m}\left(\Omega\right)^{K}$.

The space $\mathcal{ML}^{m}_{\textbf{T}}\left(\Omega\right)$
carries the initial uniform convergence structure
$\mathcal{J}_{\widehat{\textbf{T}}}$ with respect to the mapping
$\widehat{\textbf{T}}$.  That is,
\begin{eqnarray}
\mathcal{U}\in\mathcal{J}_{\widehat{\textbf{T}}}\Leftrightarrow
\widehat{\textbf{T}}\left(\mathcal{U}\right)\in\mathcal{J}_{o}^{K}
\label{JTDef}
\end{eqnarray}
Since $\widehat{\textbf{T}}$ is \textit{injective}, it follows
that $\mathcal{ML}^{m}_{\textbf{T}}\left(\Omega\right)$ is
uniformly isomorphic to the subspace $\widehat{\textbf{T}}\left(
\mathcal{ML}^{m}_{\textbf{T}}\left(\Omega\right)\right)$ of
$\mathcal{ML}^{0}\left(\Omega\right)^{K}$.  In particular,
$\widehat{\textbf{T}}$ is a uniformly continuous embedding.  In
view of the general results on the completion of uniform
convergence spaces \cite{vdWalt2}, the completion
$\mathcal{NL}_{\textbf{T}}\left(\Omega\right)$ of
$\mathcal{ML}^{m}_{\textbf{T}}\left(\Omega\right)$ is homeomorphic
to a subspace of $\mathcal{NL}\left(\Omega\right)^{K}$.  Indeed,
$\widehat{\textbf{T}}$ extends to a uniformly continuous embedding
of $\mathcal{NL}_{\textbf{T}}\left(\Omega\right)$ into
$\mathcal{NL}\left(\Omega\right)^{K}$.

\section{Approximation Results}

In this section we consider the approximation results underlying
the Order Completion Method \cite{Obergugenberger and Rosinger}
and the pseudo-topological version of that theory developed in
\cite{vdWalt5}.  In this regard we again consider a system of
nonlinear PDEs of the form (\ref{SystPDE}) through
(\ref{SystPDEDef}). Recall \cite{vdWalt5} that, subject to the
mild assumption (\ref{Assumption}), we obtain the following
\textit{local} approximation result.  We include the proof as an
illustration of the technique used. Moreover, it serves to clarify
the arguments that lead to a refinement of these results.
\begin{proposition}\label{LocApprox}
Consider a system of PDEs of the form (\ref{SystPDE}) through
(\ref{SystPDEDef}) that also satisfies (\ref{Assumption}).  Then
\begin{eqnarray}
\begin{array}{ll}
\forall & x_{0}\in\Omega\mbox{ :} \\
\forall & \epsilon>0\mbox{ :} \\
\exists & \delta>0\mbox{, }P_{1},...,P_{K}\mbox{ polynomial
in $x\in\mathbb{R}^{n}$ :} \\
& x\in B\left(x_{0},\delta\right)\cap\Omega\mbox{, }1\leq i\leq K
\Rightarrow f_{i}\left(x\right)-\epsilon<
T_{i}\left(x,D\right)\textbf{P}\left(x\right)<
f_{i}\left(x\right) \\
\end{array}\label{ALemEq}
\end{eqnarray}
Here $\textbf{P}$ is the $K$-dimensional vector valued function
with components $P_{1},...,P_{K}$.
\end{proposition}
\begin{proof}
For any $x_{0}\in\Omega$ and $\epsilon>0$ small enough it follows
by (\ref{Assumption}) that there exist
\begin{eqnarray}
\xi_{i\alpha}\in\mathbb{R}\mbox{ with }i=1,...,K\mbox{ and
}|\alpha|\leq m\nonumber
\end{eqnarray}
such that
\begin{eqnarray}
F_{i}\left(x_{0},...,\xi_{i\alpha},...\right)=f_{i}\left(x_{0}\right)-\frac{\epsilon}{2}\nonumber
\end{eqnarray}
Now take $P_{1},...,P_{K}$ polynomials in $x\in\mathbb{R}^{n}$
that satisfy
\begin{eqnarray}
D^{\alpha}P_{i}\left(x_{0}\right)=\xi_{i\alpha}\mbox{ for
}i=1,...,K\mbox{ and }|\alpha|\leq m\nonumber
\end{eqnarray}
Then it is clear that
\begin{eqnarray}
T_{i}\left(x,D\right)\textbf{P}\left(x_{0}\right)-f_{i}\left(x_{0}\right)=-\frac{\epsilon}{2}\nonumber
\end{eqnarray}
where $\textbf{P}$ is the $K$-dimensional vector valued function
on $\mathbb{R}^{n}$ with components $P_{1},...,P_{K}$.  Hence
(\ref{ALemEq}) follows by the continuity of the $f_{i}$ and the
$F_{i}$.
\end{proof}\\ \\
It should be observed that, in contradistinction to the usual
functional analytic methods, the local \textit{lower solution} in
Proposition \ref{LocApprox} is constructed in a particularly
simple way. Indeed, it is obtained by nothing but a
straightforward application of the continuity of the mapping
$\textbf{F}$.  Using exactly these same techniques, one may prove
the existence of the corresponding approximate \textit{upper
solutions}.
\begin{proposition}\label{LocApproxII}
Consider a system of PDEs of the form (\ref{SystPDE}) through
(\ref{SystPDEDef}) that also satisfies (\ref{Assumption}).  Then
\begin{eqnarray}\label{ALemEq1}
\begin{array}{ll}
\forall & x_{0}\in\Omega\mbox{ :} \\
\forall & \epsilon>0\mbox{ :} \\
\exists & \delta>0\mbox{, }P_{1},...,P_{K}\mbox{ polynomial
in $x\in\mathbb{R}^{n}$ :} \\
& x\in B\left(x_{0},\delta\right)\cap\Omega\mbox{, }1\leq i\leq k
\Rightarrow f_{i}\left(x\right)<
T_{i}\left(x,D\right)\textbf{P}\left(x\right)<
f_{i}\left(x\right)+\epsilon \\
\end{array}\nonumber
\end{eqnarray}
Here $\textbf{P}$ is the $K$-dimensional vector valued function
with components $P_{1},...,P_{K}$.
\end{proposition}
The \textit{global} approximations, corresponding to the local
approximation constructed in Propositions \ref{LocApprox} and
\ref{LocApproxII}, may be formulated as follows.
\begin{theorem}\label{GApprox}
Consider a system of PDEs of the form (\ref{SystPDE}) through
(\ref{SystPDEDef}) that also satisfies (\ref{Assumption}).  For
every $\epsilon>0$ there exists a closed nowhere dense set
$\Gamma_{\epsilon}\subset\Omega$, and functions
$\textbf{U}_{\epsilon},\textbf{V}_{\epsilon}\in\mathcal{C}^{m}
\left(\Omega\setminus\Gamma_{\epsilon}\right)^{K}$ with components
$U_{\epsilon,1},...,U_{\epsilon,K}$ and
$V_{\epsilon,1},...,V_{\epsilon,K}$ respectively, such that
\begin{equation}\label{ApEq}
f_{i}\left(x\right)-\epsilon<
T_{i}\left(x,D\right)\textbf{U}_{\epsilon}\left(x\right)<
f_{i}\left(x\right)< T_{i}\left(x,D\right)\textbf{V}_{\epsilon}
\left(x\right)< f_{i}\left(x\right)+\epsilon\mbox{,
}x\in\Omega\setminus\Gamma_{\epsilon}
\end{equation}
\end{theorem}
Once again, and as was mentioned in connection with Proposition
\ref{LocApprox}, the approximation result above is based
\textit{solely} on the existence of a compact tiling of any open
subset $\Omega$ of $\mathbb{R}^{n}$, the properties of compact
subsets of $\mathbb{R}^{n}$ and the continuity of usual real
valued functions, see for instance \cite[Theorem 4]{vdWalt5}.
Hence it makes no use of so called \textit{advanced mathematics}.
In particular, techniques from functional analysis are not used at
all.  Instead, the relevant techniques belong rather to the
classical theory of real functions.

As an immediate application of Theorem \ref{GApprox}, we may
construct a sequence $\left(\textbf{u}_{n}\right)$ in
$\mathcal{ML}^{m}\left(\Omega\right)^{K}$ so that $\left(
\textbf{Tu}_{n}\right)$ converges to $\textbf{f}$ in
$\mathcal{ML}^{0}\left(\Omega\right)^{K}$.  However, Theorem
\ref{GApprox} makes no claim concerning the convergence of the
sequence $\left(\textbf{u}_{n}\right)$, or lack thereof, in
$\mathcal{ML}^{m}\left(\Omega\right)^{K}$.  Indeed, assuming only
that (\ref{Assumption}) is satisfied, it is possible, in almost
all cases of applicative interest, to construct a sequence
$\left(\textbf{U}_{n}\right)$ that satisfies Theorem
\ref{GApprox}, and is unbounded on every neighborhood of every
point of $\Omega$.  This follows easily from the fact that, in
general, for a fixed $x_{0}\in\Omega$, the sets
\begin{eqnarray}
\{\xi\in\mathbb{R}^{M}\mbox{ : }\textbf{F}\left(x_{0},\xi\right)=
\textbf{f}\left(x_{0}\right)\}\nonumber
\end{eqnarray}
may be unbounded.

In view of the above remarks, it appears that a stronger
assumption than (\ref{Assumption}) may be required in order to
construct generalized solutions to (\ref{SystPDE}) in
$\mathcal{NL}^{m}\left(\Omega\right)^{K}$.  When formulating such
an appropriate condition on the system of PDEs (\ref{SystPDE}),
one should keep in mind that the Order Completion Method
\cite{Obergugenberger and Rosinger}, and in particular the
pseudo-topological version of the theory developed in
\cite{vdWalt3} and \cite{vdWalt5}, is based on some basic
topological processes, namely, the completion of uniform
convergence spaces, and the simple condition (\ref{Assumption}),
which is formulated entirely in terms of the usual real mappings
$\textbf{F}$ and $\textbf{f}$.  In particular, (\ref{Assumption})
does not involve any topological structures on function spaces, or
mappings on such spaces.  Furthermore, other than the mere
continuity of the mapping $\textbf{F}$, (\ref{Assumption}) places
no restriction on the \textit{type} of equation treated.  As such,
it is then clear that any further assumptions that we may wish to
impose on the system of equations (\ref{SystPDE}) in order to
obtain generalized solutions in $\mathcal{NL}^{m}\left(
\Omega\right)^{K}$ should involve only basic topological
properties of the mapping $\textbf{F}$, and should not involve any
restrictions on the type of equations.

In formulating such a condition on the system of PDEs
(\ref{SystPDE}) that will ensure the existence of a generalized
solution in $\mathcal{NL}^{m}\left(\Omega\right)^{K}$, it is
helpful to first understand more completely the role of the
condition (\ref{Assumption}) in the proof of the local
approximation result Proposition \ref{LocApprox}.  In particular,
and as is clear from the proof of Proposition \ref{LocApprox}, the
condition (\ref{Assumption}) relates to the continuity of the
mapping $\textbf{F}$.  Furthermore, and as has already been
mentioned, the approximations constructed in Proposition
\ref{LocApprox} and Theorem \ref{GApprox} concern only convergence
in the target space of the operator $\textbf{T}$ associated with
(\ref{SystPDE}).  Our interest here lies in constructing suitable
approximations in the domain of $\textbf{T}$, and as such,
properties of the inverse of the mapping $\textbf{F}$ may prove to
be particularly useful.  In view of these remarks, we introduce
the following condition.
\begin{eqnarray}
\begin{array}{ll}
\forall & x_{0}\in\Omega\mbox{ :} \\
\exists & \xi\left(x_{0}\right)\in\mathbb{R}^{M}\mbox{ :} \\
\exists & V\in\mathcal{V}_{x_{0}}\mbox{, }W\in\mathcal{V}_{\xi\left(x_{0}\right)}\mbox{ :} \\
& \begin{array}{ll}
1) & \textbf{F}\left(x_{0},\xi\left(x_{0}\right)\right)=\textbf{f}\left(x_{0}\right)\mbox{ :} \\
2) & \textbf{F}: V\times W\rightarrow \mathbb{R}^{K} \mbox{ open} \\
\end{array} \\
\end{array}\label{AssumptionI}
\end{eqnarray}
Note that the condition (\ref{AssumptionI}) above, although more
restrictive than (\ref{Assumption}), allows for the treatment of a
large class of equations.  In particular, each equation of the
form
\begin{eqnarray}
D_{t}\textbf{u}\left(x,t\right) +\textbf{G}\left(x,t,
\textbf{u}\left(x,t\right),...,D^{\alpha}_{x}\textbf{u}
\left(x,t\right),...\right)=\textbf{f}\left(x,t\right) \nonumber
\end{eqnarray}
with the mapping $\textbf{G}$ merely continuous, satisfies
(\ref{AssumptionI}).  Indeed, the mapping $\textbf{F}$ that
defines the equation through (\ref{SystPDE}) is both open and
surjective.  Other classes of equations that satisfy
(\ref{AssumptionI}) can be easily exhibited.

\section{Existence of Generalized Solutions}

The basic existence result for the Order Completion Method
\cite{vdWalt5} is an application of the global approximation
result in Theorem \ref{GApprox}, and the commutative diagram\\
\\
\\
\begin{math}
\setlength{\unitlength}{1cm} \thicklines
\begin{picture}(13,6)

\put(3.4,5.4){$\mathcal{ML}^{m}_{\textbf{T}}\left(\Omega\right)$}
\put(5.1,5.5){\vector(1,0){6.4}}
\put(11.6,5.4){$\mathcal{ML}^{0}\left(\Omega\right)^{K}$}
\put(7.8,5.9){$\widehat{\textbf{T}}$}
\put(3.8,5.2){\vector(0,-1){3.5}} \put(5.0,1.2){\vector(1,0){6.5}}
\put(3.4,1.1){$\mathcal{NL}_{\textbf{T}}\left(\Omega\right)$}
\put(11.6,1.1){$\mathcal{NL}\left(\Omega\right)^{K}$}
\put(3.5,3.4){$\phi$} \put(12.2,3.4){$\varphi$}
\put(12.0,5.2){\vector(0,-1){3.5}}
\put(7.8,1.4){$\widehat{\textbf{T}}^{\sharp}$}

\end{picture}
\end{math}\\
\\
Here $\phi$ and $\varphi$ are the uniformly continuous embeddings
associated canonically with the completions
$\mathcal{NL}_{\textbf{T}}\left(\Omega\right)$ and
$\mathcal{NL}\left(\Omega\right)$, and
$\widehat{\textbf{T}}^{\sharp}$ is the extension of
$\widehat{\textbf{T}}$ achieved through uniform continuity.  The
approximation result, Theorem \ref{GApprox}, is used to construct
a Cauchy sequence in
$\mathcal{ML}^{m}_{\textbf{T}}\left(\Omega\right)$ so that its
image under $\widehat{\textbf{T}}$ converges to $\textbf{f}$. This
delivers the \textit{existence} of a solution to the generalized
equation
\begin{eqnarray}
\widehat{\textbf{T}}^{\sharp}\textbf{U}^{\sharp}=\textbf{f}\label{GenEq}
\end{eqnarray}
which we interpret as a generalized solution to (\ref{SystPDE}).
 Moreover, since $\widehat{\textbf{T}}$ is a uniformly
continuous embedding, so is its extension
$\widehat{\textbf{T}}^{\sharp}$ and hence the solution to
(\ref{GenEq}) is \textit{unique}.
\begin{theorem}\label{SolEx}
For every $\textbf{f}\in\mathcal{C}^{0}\left(\Omega\right)^{K}$
that satisfies (\ref{Assumption}), there exists a unique
$\textbf{U}^{\sharp}\in \mathcal{NL}_{\textbf{T}}
\left(\Omega\right)$ so that
\begin{equation}\label{ExtPDEII}
\widehat{\textbf{T}}^{\sharp}\textbf{U}^{\sharp}=\textbf{f}
\end{equation}
\end{theorem}
The aim of this section is to obtain the existence of generalized
solutions to (\ref{SystPDE}) in the space
$\mathcal{NL}^{m}\left(\Omega\right)^{K}$.  In order to formulate
an extended version of the equation (\ref{SystPDE}) in this
setting, the partial differential operator $\textbf{T}$ must be
extended to the completion of
$\mathcal{ML}^{m}\left(\Omega\right)^{K}$.  The mapping
$\textbf{T}$ must therefore be uniformly continuous with respect
to the uniform convergence structures on
$\mathcal{ML}^{m}\left(\Omega\right)^{K}$ and
$\mathcal{ML}^{0}\left(\Omega\right)^{K}$.  The proof of this
result is deferred to the appendix.
\begin{theorem}\label{TUnifCon}
The mapping
\begin{eqnarray}
\textbf{T}:\mathcal{ML}^{m}\left(\Omega\right)^{K} \rightarrow
\mathcal{ML}^{0}\left(\Omega\right)^{K}\nonumber
\end{eqnarray}
defined in (\ref{SystPDEMap}) to (\ref{SystPDEMapComp}) is
uniformly continuous.
\end{theorem}
In view of Theorem \ref{TUnifCon} the mapping $\textbf{T}$ extends
uniquely to a uniformly continuous mapping
\begin{eqnarray}
\textbf{T}^{\sharp}:\mathcal{NL}^{m}\left(\Omega\right)^{K}
\rightarrow \mathcal{NL}\left(\Omega\right)^{K}\label{ExtPDEMapI}
\end{eqnarray}
so that the diagram\\
\\
\\
\begin{math}
\setlength{\unitlength}{1cm} \thicklines
\begin{picture}(13,6)

\put(3.4,5.4){$\mathcal{ML}^{m}\left(\Omega\right)^{K}$}
\put(5.1,5.5){\vector(1,0){6.4}}
\put(11.6,5.4){$\mathcal{ML}^{0}\left(\Omega\right)^{K}$}
\put(7.8,5.9){$\textbf{T}$} \put(3.8,5.2){\vector(0,-1){3.5}}
\put(5.0,1.2){\vector(1,0){6.5}}
\put(3.4,1.1){$\mathcal{NL}^{m}\left(\Omega\right)^{K}$}
\put(11.6,1.1){$\mathcal{NL}\left(\Omega\right)^{K}$}
\put(3.5,3.4){$\psi$} \put(12.2,3.4){$\varphi$}
\put(12.0,5.2){\vector(0,-1){3.5}}
\put(7.8,1.4){$\textbf{T}^{\sharp}$}

\end{picture}
\end{math}\\
\\
commutes, with $\psi$ and $\varphi$ the uniformly continuous
embeddings associated with the completions of
$\mathcal{ML}^{m}\left(\Omega\right)^{K}$ and
$\mathcal{ML}^{0}\left(\Omega\right)^{K}$, respectively. Therefore
we are justified in formulating the generalized equation
\begin{eqnarray}
\textbf{T}^{\sharp}\textbf{u}^{\sharp}=\textbf{f}\label{ExtSystPDE}
\end{eqnarray}
where the unknown $\textbf{u}^{\sharp}$ ranges over
$\mathcal{NL}^{m}\left(\Omega\right)^{K}$.
\begin{theorem}\label{SolExIII}
For every $\textbf{f}\in\mathcal{C}^{0}\left(\Omega\right)^{K}$ so
that the system of PDEs (\ref{SystPDE}) satisfies
(\ref{AssumptionI}), there is some $\textbf{u}^{\sharp}\in
\mathcal{NL}^{m} \left(\Omega\right)^{K}$ so that
\begin{equation}\label{ExtPDEIV}
\textbf{T}^{\sharp}\textbf{u}^{\sharp}=\textbf{f}
\end{equation}
\end{theorem}
\begin{proof}
Set
\begin{equation}
\Omega=\bigcup_{\nu\in\mathbb{N}}C_{\nu}\nonumber
\end{equation}
where, for $\nu\in\mathbb{N}$, the compact sets $C_{\nu}$ are
$n$-dimensional intervals
\begin{equation}
C_{\nu}=[a_{\nu},b_{\nu}]\nonumber
\end{equation}
with $a_{\nu}=\left(a_{\nu,1},...,a_{\nu,n}\right)$,
$b_{\nu}=\left(b_{\nu,1},...,b_{\nu,n}\right)\in\mathbb{R}^{n}$
and $a_{\nu,j}\leq b_{\nu,j}$ for every $j=1,...,n$.  We also
assume that the $C_{\nu}$, with $\nu\in\mathbb{N}$ are locally
finite, that is,
\begin{equation}\label{AE2}
\begin{array}{ll}
\forall & x\in\Omega\mbox{ :} \\
\exists & V\subseteq\Omega\mbox{ a neighborhood of $x$ :} \\
& \{\nu\in\mathbb{N}\mbox{ : }C_{\nu}\cap
V\neq\emptyset\}\mbox{ is finite} \\
\end{array}\nonumber
\end{equation}
We also assume that the interiors of $C_{\nu}$, with
$\nu\in\mathbb{N}$, are pairwise disjoint.\\
Let $C_{\nu}$ be arbitrary but fixed.  In view of
(\ref{AssumptionI}) and the continuity of $\textbf{f}$, we have
\begin{eqnarray}
\begin{array}{ll}
\forall & x_{0}\in C_{\nu}\mbox{ :} \\
\exists & \xi\left(x_{0}\right)\in\mathbb{R}^{M}\mbox{, }\textbf{F}\left(x_{0},\xi\left(x_{0}\right)\right)=\textbf{f}\left(x_{0}\right)\mbox{ :} \\
\exists & \delta\mbox{, }\epsilon>0\mbox{ :} \\
& \begin{array}{ll} 1) &
\left\{\left(x,\textbf{f}\left(x\right)\right)\mbox{ :
}\|x-x_{0}\|<\delta \right\}\subset
\textrm{int}\left\{\left(x,\textbf{F}\left(x,\xi\right)\right)\begin{array}{|l}
\|x-x_{0}\|<\delta \\
\|\xi-\xi\left(x_{0}\right)\|<\epsilon \\
\end{array}\right\} \\
2) & \textbf{F}: B_{\delta}\left(x_{0}\right)\times B_{2\epsilon}\left(\xi\left(x_{0}\right)\right) \rightarrow \mathbb{R}^{K}\mbox{ open} \\
\end{array} \\
\end{array}\label{LocSolI}
\end{eqnarray}
For each $x_{0}\in C_{\nu}$, fix
$\xi\left(x_{0}\right)\in\mathbb{R}^{M}$ in (\ref{LocSolI}).
Since $C_{\nu}$ is compact, it follows from (\ref{LocSolI}) that
\begin{eqnarray}
\begin{array}{ll}
\exists & \delta>0\mbox{ :} \\
\forall & x_{0}\in C_{\nu}\mbox{ :} \\
\exists & \epsilon>0\mbox{ :} \\
& \begin{array}{ll} 1) &
\left\{\left(x,\textbf{f}\left(x\right)\right)\mbox{ :
}\|x-x_{0}\|<\delta \right\}\subset
\textrm{int}\left\{\left(x,\textbf{F}\left(x,\xi\right)\right)\begin{array}{|l}
\|x-x_{0}\|<\delta \\
\|\xi-\xi\left(x_{0}\right)\|<\epsilon \\
\end{array}\right\} \\
2) & \textbf{F}: B_{\delta}\left(x_{0}\right)\times B_{2\epsilon}\left(\xi\left(x_{0}\right)\right) \rightarrow \mathbb{R}^{K}\mbox{ open} \\
\end{array} \\
\end{array}\label{LocSolII}
\end{eqnarray}
Subdivide $C_{\nu}$ into $n$-dimensional intervals
$I_{\nu,1},...,I_{\nu,\mu_{\nu}}$ with diameter not exceeding
$\delta$ such that their interiors are pairwise disjoint.  If
$a_{\nu,j}$ with $j=1,...,\mu_{\nu}$ is the center of the interval
$I_{\nu,j}$ then by (\ref{LocSolII}) we have
\begin{eqnarray}
\begin{array}{ll}
\forall & j=1,...,\mu_{\nu}\mbox{ :} \\
\exists & \epsilon_{\nu,j}>0\mbox{ :} \\
& \begin{array}{ll} 1) &
\left\{\left(x,\textbf{f}\left(x\right)\right)\mbox{ : }x\in
I_{\nu,j} \right\}\subset
\textrm{int}\left\{\left(x,\textbf{F}\left(x,\xi\right)\right)\begin{array}{|l}
x\in I_{\nu,j} \\
\|\xi-\xi\left(a_{\nu,j}\right)\|<\epsilon_{\nu,j} \\
\end{array}\right\} \\
2) & \textbf{F}: I_{\nu,j}\times B_{2\epsilon_{\nu,j}}\left(\xi\left(a_{\nu,j}\right)\right) \rightarrow \mathbb{R}^{K}\mbox{ open} \\
\end{array} \\
\end{array}\label{LocSolIII}
\end{eqnarray}
Take $\gamma>0$ arbitrary but fixed.  In view of Proposition
\ref{LocApprox} and (\ref{LocSolIII}), we have
\begin{eqnarray}
\begin{array}{ll}
\forall & x_{0}\in I_{\nu,j}\mbox{ :} \\
\exists & \textbf{U}_{x_{0}}=\textbf{U}\in\mathcal{C}^{m}\left(\mathbb{R}^{n}\right)^{K}\mbox{ :} \\
\exists & \delta=\delta_{x_{0}}>0\mbox{ :} \\
\forall & i=1,...,K\mbox{ :} \\
& x\in B_{\delta}\left(x_{0}\right)\cap
I_{\nu,j}\Rightarrow\left(\begin{array}{ll}
1) &  D^{\alpha}U_{i}\left(x\right)\in B_{\epsilon_{\nu,j}}\left(\xi\left(a_{\nu,j}\right)\right)\mbox{, }|\alpha|\leq m \\
2) & f_{i}\left(x\right)-\gamma< T_{i}\left(x,D\right)\textbf{U}\left(x\right)<f_{i}\left(x\right) \\
\end{array}\right) \\
\end{array}\nonumber
\end{eqnarray}
As above, we may subdivide $I_{\nu,j}$ into pairwise disjoint,
$n$-dimensional intervals $J_{\nu,j,1},...,J_{\nu,j,\mu_{\nu,j}}$
so that for $k=1,...,\mu_{\nu,j}$ we have
\begin{eqnarray}
\begin{array}{ll}
\exists & \textbf{U}^{\nu,j,k}=\textbf{U}\in\mathcal{C}^{m}\left(\mathbb{R}^{n}\right)^{K}\mbox{ :} \\
\forall & i=1,...,K\mbox{ :} \\
& x\in J_{\nu,j,k}\Rightarrow\left(\begin{array}{ll}
1) &  D^{\alpha}U_{i}\left(x\right)\in B_{\epsilon_{\nu,j}}\left(\xi\left(a_{\nu,j}\right)\right)\mbox{, }|\alpha|\leq m \\
2) & f_{i}\left(x\right)-\gamma< T_{i}\left(x,D\right)\textbf{U}\left(x\right)<f_{i}\left(x\right) \\
\end{array}\right) \\
\end{array}\label{AA1}
\end{eqnarray}
Set
\begin{eqnarray}
\textbf{V}_{1}=\sum_{\nu\in\mathbb{N}}\left(\sum_{j=1}^{\mu_{\nu}}
\left(\sum_{k=1}^{\mu_{\nu,j}}\chi_{J_{\nu,j,k}}\textbf{U}_{\nu,j,k}\right)\right)
\nonumber
\end{eqnarray}
where $\chi_{J_{\nu,j,k}}$ is the characteristic function of
$J_{\nu,j,k}$, and
\begin{eqnarray}
\Gamma_{1}=\Omega\setminus\left(\bigcup_{\nu\in\mathbb{N}}
\left(\bigcup_{j=1}^{\mu_{\nu}}\left(\bigcup_{k=1}^{\mu_{\nu,j}}
\textrm{int}J_{\nu,j,k} \right)\right)\right).
\end{eqnarray}
Then $\Gamma_{1}$ is closed nowhere dense, and $\textbf{V}_{1}\in
\mathcal{C}^{m}\left(\Omega\setminus\Gamma_{1}\right)^{K}$.  In
view of (\ref{AA1}) we have, for each $i=1,...,K$
\begin{eqnarray}
f_{i}\left(x\right)-\gamma<
T_{i}\left(x,D\right)\textbf{V}_{1}\left(x\right)<
f_{i}\left(x\right)\mbox{, }x\in\Omega\setminus\Gamma_{1}\nonumber
\end{eqnarray}
Furthermore, for each $\nu\in\mathbb{N}$, for each
$j=1,...,\mu_{\nu}$, each $k=1,...,\mu_{\nu,j}$, each
$|\alpha|\leq m$ and every $i=1,...,K$ we have
\begin{eqnarray}
x\in \textrm{int}J_{\nu,j,k} \Rightarrow
\xi^{\alpha}_{i}\left(a_{\nu,j}\right)-\epsilon<
D^{\alpha}V_{1,i}\left(x\right)<
\xi^{\alpha}_{i}\left(a_{j}\right)+\epsilon \nonumber
\end{eqnarray}
Therefore the functions $\lambda_{1,i}^{\alpha},\mu_{1,i}^{\alpha}
\in \mathcal{C}^{0}\left(\Omega\setminus\Gamma_{1}\right)$ defined
as
\begin{eqnarray}
\lambda_{1,i}^{\alpha}\left(x\right) =
\xi^{\alpha}_{i}\left(a_{j}\right)-2\epsilon_{\nu,j}\mbox{ if }
x\in \textrm{int}I_{\nu,j}\nonumber
\end{eqnarray}
and
\begin{eqnarray}
\mu_{1,i}^{\alpha}\left(x\right) =
\xi^{\alpha}_{i}\left(a_{j}\right)+2\epsilon_{\nu,j}\mbox{ if }
x\in \textrm{int}I_{\nu,j}\nonumber
\end{eqnarray}
satisfies
\begin{eqnarray}
\lambda_{1,i}^{\alpha}\left(x\right) <
D^{\alpha}V_{1,i}\left(x\right) <
\mu_{1,i}^{\alpha}\left(x\right)\mbox{,
}x\in\Omega\setminus\Gamma_{1}\nonumber
\end{eqnarray}
and
\begin{eqnarray}
\mu_{1,i}^{\alpha}\left(x\right) -
\lambda_{1,i}^{\alpha}\left(x\right) <4\epsilon_{\nu,j}\mbox{,
}x\in\textrm{int}I_{\nu,j} \nonumber
\end{eqnarray}
Applying (\ref{LocSolIII}), and proceeding in a fashion similar as
above, we may construct, for each $n\in\mathbb{N}$ such that
$n>1$, a closed nowhere dense set $\Gamma_{n}\subset\Omega$, a
function
$\textbf{V}_{n}\in\mathcal{C}^{m}\left(\Omega\setminus\Gamma_{n}
\right)^{K}$ and functions $\lambda_{n,i}^{\alpha},
\mu_{n,i}^{\alpha}\in\mathcal{C}^{0}\left(\Omega\setminus\Gamma_{n}
\right)$ so that, for each $i=1,...,K$ and $|\alpha|\leq m$
\begin{eqnarray}
f_{i}\left(x\right)-\frac{\gamma}{n}<
T_{i}\left(x,D\right)\textbf{V}_{n}\left(x\right)<
f_{i}\left(x\right)\mbox{,
}x\in\Omega\setminus\left(\Gamma_{n}\cup \Gamma_{n-1}\right).
\label{EQ1}
\end{eqnarray}
Furthermore,
\begin{eqnarray}
\lambda_{n-1,i}^{\alpha}\left(x\right)<\lambda_{n,i}^{\alpha}\left(x\right)
< D^{\alpha}V_{n,i} \left(x\right)<
\mu_{n,i}^{\alpha}\left(x\right)<\mu_{n-1,i}^{\alpha}\left(x\right)\mbox{,
}x\in\Omega\setminus\Gamma_{n}\label{EQ2}
\end{eqnarray}
and
\begin{eqnarray}
\mu_{n,i}^{\alpha}\left(x\right) -
\lambda_{n,i}^{\alpha}\left(x\right) <
\frac{4\epsilon_{\nu,j}}{n}\mbox{, }x\in
\left(\textrm{int}I_{\nu,j}\right)\cap\left(\Omega\setminus\Gamma_{n}\right).
\label{EQ3}
\end{eqnarray}
For each $n\in\mathbb{N}$ and $i=1,...,K$ set $v_{n,i}
=\left(I\circ S\right)\left(V_{n,i}\right)$.  Furthermore, for
each $|\alpha|\leq m$ set $\overline{\lambda}_{n,i}^{\alpha}
=\left(I\circ S\right)\left(\lambda_{n,i}^{\alpha}\right)$ and
$\overline{\mu}_{n,i}'^{\alpha} =\left(I\circ
S\right)\left(\mu_{n,i}^{\alpha}\right)$.  Then, in view of
(\ref{EQ1}) we have
\begin{eqnarray}
\begin{array}{ll}
\forall & n\in\mathbb{N}\mbox{ :} \\
\forall & i=1,...,K\mbox{ :} \\
& f_{i}-\frac{\gamma}{n}\leq T_{i}\textbf{v}_{n}\leq f_{i} \\
\end{array}\nonumber
\end{eqnarray}
and from (\ref{EQ2}) it follows that
\begin{eqnarray}
\begin{array}{ll}
\forall & n\in\mathbb{N}\mbox{ :} \\
\forall & i=1,...,K\mbox{ :} \\
\forall & |\alpha|\leq m\mbox{ :} \\
& \overline{\lambda}_{n,i}^{\alpha}\leq \overline{\lambda}_{n+1,i}^{\alpha}\leq \mathcal{D}^{\alpha}v_{n,i}\leq \overline{\mu}_{n+1,i}^{\alpha}\leq\overline{\mu}_{n,i}^{\alpha} \\
\end{array}.\nonumber
\end{eqnarray}
Moreover, in view of (\ref{EQ3}) it follows that
\begin{eqnarray}
\overline{\mu}_{n,i}^{\alpha}\left(x\right)-
\overline{\lambda}_{n,i}^{\alpha}\left(x\right)
<\frac{4\epsilon_{\nu,j}}{n}\mbox{, }x\in \textrm{int}I_{\nu,j}.
\nonumber
\end{eqnarray}
Therefore $\left(\textbf{v}_{n}\right)$ is a Cauchy sequence in
$\mathcal{ML}^{m}\left(\Omega\right)^{K}$, and
$\left(\textbf{Tv}_{n}\right)$ converges to $\textbf{f}$ in
$\mathcal{ML}^{0}\left(\Omega\right)^{K}$.  The result now follows
by Theorem \ref{TUnifCon}.
\end{proof}\\ \\

\section{The Structure of Generalized Functions}

Recall \cite{vdWalt5} that, in view of the abstract construction
of the completion of a uniform convergence space \cite{Wyler}, the
unique solution to the generalized equation (\ref{GenEq}) may be
represented as the equivalence class of Cauchy filters on
$\mathcal{ML} ^{m} _{\textbf{T}} \left(\Omega\right)$
\begin{eqnarray}
\left\{\mathcal{F}\mbox{ a filter on
}\mathcal{ML}^{m}_{\textbf{T}}\left(\Omega\right)\mbox{ : }
\widehat{\textbf{T}}\left(\mathcal{F}\right)\mbox{ converges to
}\textbf{f} \right\}\label{SolChar}
\end{eqnarray}
That is, it consists of the totality of all filters $\mathcal{F}$
on $\mathcal{ML}^{m}_{\textbf{T}}\left(\Omega\right)$ so that
$\widehat{\textbf{T}}\left(\mathcal{F}\right)$ converges to
$\textbf{f}$ in $\mathcal{ML}^{0}\left(\Omega\right)^{K}$.
Moreover, each classical solution
$\textbf{u}\in\mathcal{C}^{m}\left(\Omega\right)^{K}$, and also
each nonclassical solution
$\textbf{u}\in\mathcal{ML}^{m}\left(\Omega\right)^{K}$ to
(\ref{SystPDE}) generates a Cauchy filter in
$\mathcal{ML}^{m}_{\textbf{T}}\left(\Omega\right)$ which belongs
to the equivalence class (\ref{SolChar}).  Therefore the
generalized solution $\textbf{U}^{\sharp}\in
\mathcal{NL}_{\textbf{T}}\left(\Omega\right)$ is consistent with
the usual classical solutions as well as the nonclassical
solutions in $\textbf{u}\in\mathcal{ML} ^{m}\left(\Omega\right)
^{K}$.  This may be represented in the commutative diagram\\
\\
\\
\begin{math}
\setlength{\unitlength}{1cm} \thicklines
\begin{picture}(13,6)

\put(3.4,5.2){$\mathcal{ML}^{m}\left(\Omega\right)^{K}$}
\put(5.1,5.3){\vector(1,0){6.4}}
\put(11.6,5.2){$\mathcal{ML}^{0}\left(\Omega\right)^{K}$}
\put(7.2,0.9){$\mathcal{ML}^{m}_{\textbf{T}}\left(\Omega\right)$}
\put(7.8,5.7){$\textbf{T}$} \put(4.0,5.0){\vector(1,-1){3.5}}
\put(5.0,3.2){$q_{\textbf{T}}$}
\put(10.7,3.2){$\widehat{\textbf{T}}$}
\put(8.5,1.5){\vector(1,1){3.5}}

\end{picture}
\end{math}\\
\\
where $q_{\textbf{T}}$ is the canonical quotient map associated
with the equivalence relation (\ref{TEquiv}).  In view of this
diagram it appears that the mapping $\widehat{\textbf{T}}$ is
nothing but a \textit{representation} of the usual nonlinear
partial differential operator $\textbf{T}$. By virtue of the
uniform continuity of the mapping $\textbf{T}$, we obtain a
similar representation for the extended operator
$\textbf{T}^{\sharp}$.\\
\\
\\
\begin{math}
\setlength{\unitlength}{1cm} \thicklines
\begin{picture}(13,6)

\put(3.4,5.2){$\mathcal{NL}^{m}\left(\Omega\right)^{K}$}
\put(5.1,5.3){\vector(1,0){6.4}}
\put(11.6,5.2){$\mathcal{NL}\left(\Omega\right)^{K}$}
\put(7.2,0.9){$\mathcal{NL}_{\textbf{T}^{\sharp}}\left(\Omega\right)$}
\put(7.8,5.7){$\textbf{T}^{\sharp}$}
\put(4.0,5.0){\vector(1,-1){3.5}}
\put(5.0,3.2){$q_{\textbf{T}}^{\sharp}$}
\put(10.7,3.2){$\widehat{\textbf{T}}^{\sharp}$}
\put(8.5,1.5){\vector(1,1){3.5}}

\end{picture}
\end{math}\\
\\
In view of this diagram, every generalized solution
$\textbf{u}^{\sharp} \in \mathcal{NL}^{m}\left(\Omega\right)^{K}$
to (\ref{SystPDE}) is mapped unto the unique generalized solution
in $\mathcal{NL}_{\textbf{T}^{\sharp}}\left(\Omega\right)$.  As
such, these two concepts of generalized solution are consistent.
Furthermore,
\begin{eqnarray}
\begin{array}{ll}
\forall & \textbf{u}^{\sharp}\mbox{, }\textbf{v}^{\sharp}\in\mathcal{NL}^{m}\left(\Omega\right)^{K}\mbox{ :} \\
& \textbf{T}^{\sharp}\textbf{u}^{\sharp}=\textbf{T}^{\sharp}\textbf{v}^{\sharp}\Rightarrow q_{\textbf{T}}^{\sharp}\textbf{u}^{\sharp}=q_{\textbf{T}}^{\sharp}\textbf{v}^{\sharp} \\
\end{array}\nonumber
\end{eqnarray}
so that the space
$\mathcal{NL}_{\textbf{T}^{\sharp}}\left(\Omega\right)$ contains
the quotient space
$\mathcal{NL}_{\textbf{T}^{\sharp}}\left(\Omega\right)
/\sim_{\textbf{T}^{\sharp}}$, where
\begin{eqnarray}
\begin{array}{ll}
\forall & \textbf{u}^{\sharp},\textbf{v}^{\sharp}\in\mathcal{NL}^{m}\left(\Omega\right)^{K}\mbox{ :} \\
& \textbf{u}^{\sharp}\sim_{\textbf{T}^{\sharp}}\textbf{v}^{\sharp} \Leftrightarrow \textbf{T}^{\sharp}\textbf{u}^{\sharp}=\textbf{T}^{\sharp}\textbf{v}^{\sharp} \\
\end{array}\nonumber
\end{eqnarray}
Therefore, the generalized solution
$\textbf{U}^{\sharp}\in\mathcal{NL}_{\textbf{T}}\left(\Omega\right)$
constructed in Theorem \ref{SolEx} may be represented as
\begin{eqnarray}
\textbf{U}^{\sharp}= \left\{\textbf{u}^{\sharp}\in\mathcal{NL}^{m}
\left(\Omega\right)^{K}\mbox{ :
}\textbf{T}^{\sharp}\textbf{u}^{\sharp}=\textbf{f}
\right\}\label{GSRep}
\end{eqnarray}

Regarding the structure of the space
$\mathcal{NL}^{m}\left(\Omega\right)$ and its elements, we may
recall that the uniform convergence structure $\mathcal{J}_{D}$ on
$\mathcal{ML}^{m}\left(\Omega\right)$ is the initial uniform
convergence structure with respect to the family of mappings
\begin{eqnarray}
\left(\mathcal{D}^{\alpha}: \mathcal{ML}^{m}\left(\Omega\right)
\rightarrow
\mathcal{ML}^{0}\left(\Omega\right)\right)_{|\alpha|\leq
m}\nonumber
\end{eqnarray}
As such, the mapping
\begin{eqnarray}
\textbf{D}: \mathcal{ML}^{m}\left(\Omega\right) \ni u\mapsto
\left(\mathcal{D}^{\alpha}u\right)_{|\alpha|\leq
m}\in\mathcal{ML}^{0}\left(\Omega\right)^{M}\nonumber
\end{eqnarray}
is a uniformly continuous embedding.  In particular, for each
$|\alpha|\leq m$, the diagram\\
\\
\\
\begin{math}
\setlength{\unitlength}{1cm} \thicklines
\begin{picture}(13,6)

\put(3.4,5.2){$\mathcal{ML}^{m}\left(\Omega\right)$}
\put(5.1,5.3){\vector(1,0){6.4}}
\put(11.6,5.2){$\mathcal{ML}^{0}\left(\Omega\right)^{M}$}
\put(7.2,0.9){$\mathcal{ML}^{0}\left(\Omega\right)$}
\put(7.8,5.7){$\textbf{D}$} \put(4.0,5.0){\vector(1,-1){3.5}}
\put(4.9,3.2){$\mathcal{D}^{\alpha}$}
\put(10.7,3.2){$\pi_{\alpha}$} \put(12.1,5.0){\vector(-1,-1){3.5}}

\end{picture}
\end{math}\\
\\
commutes, with $\pi_{\alpha}$ the projection.  This diagram
amounts to a decomposition of
$u\in\mathcal{ML}^{m}\left(\Omega\right)$ into its differential
components.  In view of the uniform continuity of the mapping
$\textbf{D}$ and its inverse, $\textbf{D}$ extends to an embedding
\begin{eqnarray}
\textbf{D}^{\sharp}:\mathcal{ML}^{m}\left(\Omega\right)^{\sharp}
\rightarrow \mathcal{NL}\left(\Omega\right)^{M}\nonumber
\end{eqnarray}
Moreover, since each mapping $\mathcal{D}^{\alpha}$ is uniformly
continuous, one obtains the commutative diagram\\
\\
\\
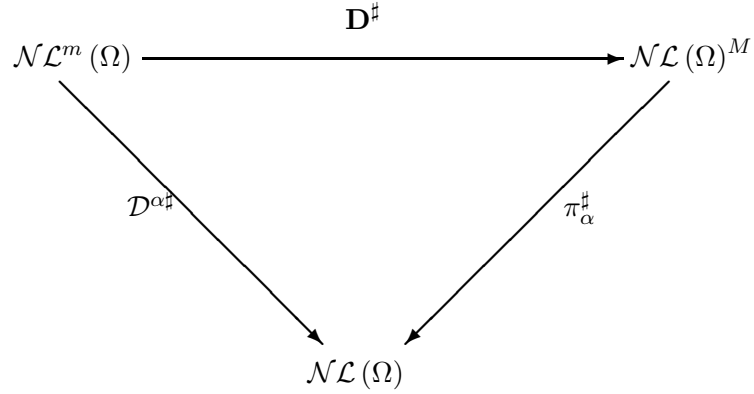
\begin{math}
\setlength{\unitlength}{1cm} \thicklines
\begin{picture}(13,6)

\put(3.4,5.2){$\mathcal{NL}^{m}\left(\Omega\right)$}
\put(5.1,5.3){\vector(1,0){6.4}}
\put(11.6,5.2){$\mathcal{NL}\left(\Omega\right)^{M}$}
\put(7.3,1.0){$\mathcal{NL}\left(\Omega\right)$}
\put(7.8,5.7){$\textbf{D}^{\sharp}$}
\put(4.0,5.0){\vector(1,-1){3.5}}
\put(4.9,3.2){$\mathcal{D}^{\alpha\sharp}$}
\put(10.7,3.2){$\pi_{\alpha}^{\sharp}$}
\put(12.1,5.0){\vector(-1,-1){3.5}}

\end{picture}
\end{math}\\
\\
where $\mathcal{D}^{\alpha\sharp}$ is the extension through
uniform continuity of the partial differential operator
$\mathcal{D}^{\alpha}$. Since the mapping $\textbf{D}^{\sharp}$ is
an embedding, and in view of the commutative diagram above, each
generalized function
$\textbf{u}^{\sharp}\in\mathcal{NL}^{m}\left(\Omega\right)$ may be
uniquely represented by its differential components
\begin{eqnarray}
\textbf{u}^{\sharp}\mapsto \textbf{D}^{\sharp}\textbf{u}^{\sharp}=
\left(\mathcal{D}^{\alpha\sharp}\textbf{u}^{\sharp}\right)_{|\alpha|\leq
m}\nonumber
\end{eqnarray}
Moreover, each differential component $\mathcal{D}^{\alpha\sharp}
\textbf{u}^{\sharp}$ of $\textbf{u}^{\sharp}$ is a nearly finite
normal lower semi-continuous function.  We note, therefore, that
the set of singular points of each
$\textbf{u}^{\sharp}\in\mathcal{NL}^{m}\left(\Omega\right)$, that
is, the set
\begin{eqnarray}
\left\{x\in\Omega\begin{array}{|ll}
\exists & |\alpha|\leq m\mbox{ :} \\
& \mathcal{D}^{\alpha\sharp}\textbf{u}^{\sharp}\mbox{ not continuous at }x \\
\end{array} \right\}\nonumber
\end{eqnarray}
is at most a set of First Baire Category.  That is, it is a
topologically small set.  However, this set may be dense in
$\Omega$.  Furthermore, such a set may have arbitrarily large
positive Lebesgue measure \cite{Oxtoby}.

\section{Conclusion}

We have established the existence of generalized solutions to a
large class of systems of nonlinear PDEs.  The space of
generalized functions that contains the solutions is constructed
as the uniform convergence space completion of a space of
piecewise smooth functions, and is independent of the particular
partial differential equation under consideration. Moreover, the
generalized functions may be represented through their
differential components as normal lower semi-continuous functions.
To what extent the generalized solutions may be interpreted
classically, that is, as continuously differentiable functions
that satisfy the equations classically, is still an open problem.

In terms of the previous existence and uniqueness results obtained
through the Order Completion Method, notably Theorem \ref{SolEx},
we may interpret Theorem \ref{SolExIII} as a regularity result.
In particular, it appear that the generalized solution delivered
by Theorem \ref{SolEx} is nothing but the \textit{totality} of all
solutions in $\mathcal{NL}^{m}\left(\Omega\right)^{K}$.

\appendix

This appendix contains the proof of Theorem \ref{TUnifCon}.  It is
based on the following results which may be found in
\cite{AngvdWalt}, \cite{vdWalt3} and \cite{vdWalt6}.
\begin{proposition}\label{MLDense}
For any $u\in\mathcal{NL}\left(\Omega\right)$ there exists
sequences $\left(\lambda_{n}\right)$ and $\left(\mu_{n}\right)$ in
$\mathcal{ML}^{0}\left(\Omega\right)$ so that
\begin{eqnarray}
\begin{array}{ll}
\forall & n\in\mathbb{N}\mbox{ :} \\
& \lambda_{n}\leq \lambda_{n+1}\leq u \leq \mu_{n+1}\leq \mu_{n} \\
\end{array}
\end{eqnarray}
and
\begin{eqnarray}
\begin{array}{ll}
\forall & x\in\Omega\mbox{ :} \\
& \sup\{\lambda_{n}\left(x\right)\mbox{ : }n\in\mathbb{N}\}=u\left(x\right)=\inf\{\mu_{n}\left(x\right)\mbox{ : }n\in\mathbb{N}\} \\
\end{array}\nonumber
\end{eqnarray}
\end{proposition}
\begin{proposition}\label{NLBounded}
Consider a set $U\subset \mathcal{NL}\left(\Omega\right)$ so that
\begin{eqnarray}
\begin{array}{ll}
\exists & B\subseteq \Omega\mbox{ of First Baire Category :} \\
\exists & v:\Omega\setminus B\rightarrow \mathbb{R}\mbox{ :} \\
& x\in\Omega\setminus B\Rightarrow u\left(x\right)\leq v\left(x\right)\mbox{, }u\in U \\
\end{array}\nonumber
\end{eqnarray}
Then there is some $w\in\mathcal{NL}\left(\Omega\right)$ so that
\begin{eqnarray}
\begin{array}{ll}
\forall & u\in U\mbox{ :} \\
& u\leq w \\
\end{array}\nonumber
\end{eqnarray}
The corresponding statement for sets bounded from below is also
true.
\end{proposition}
\begin{proposition}
Let $L$ be a lattice with respect to a given partial order $\leq$.
\begin{enumerate}
    \item For every $n\in\mathbb{N}$, let the sequence
    $\left(u_{m,n}\right)$ in $L$ be bounded and increasing and
    let
    \begin{eqnarray}
    \begin{array}{l}
u_{n}=\sup\{u_{m,n}\mbox{ : }m\in\mathbb{N}\}\mbox{,
}n\in\mathbb{N} \\
u_{n}'=\sup\{u_{m,n}\mbox{ : }m=1,...,n\} \\
\end{array}\nonumber
    \end{eqnarray}
    If the sequence $\left(u_{n}\right)$ is bounded from above and
    increasing, and has supremum in $L$, then the sequence
    $\left(u'_{n}\right)$ is bounded and increasing and
    \begin{eqnarray}
\sup\{u_{n}\mbox{ : }n\in\mathbb{N}\}= \sup\{u'_{n}\mbox{ :
}n\in\mathbb{N}\}\nonumber
    \end{eqnarray}
    \item For every $n\in\mathbb{N}$, let the sequence
    $\left(v_{m,n}\right)$ in $L$ be bounded and decreasing and
    let
    \begin{eqnarray}
    \begin{array}{l}
v_{n}=\inf\{v_{m,n}\mbox{ : }m\in\mathbb{N}\}\mbox{,
}n\in\mathbb{N} \\
v_{n}'=\inf\{v_{m,n}\mbox{ : }m=1,...,n\} \\
\end{array}\nonumber
    \end{eqnarray}
    If the sequence $\left(v_{n}\right)$ is bounded from below and
    decreasing, and has infimum in $L$, then the sequence
    $\left(v'_{n}\right)$ is bounded and decreasing and
    \begin{eqnarray}
\inf\{v_{n}\mbox{ : }n\in\mathbb{N}\}= \inf\{v'_{n}\mbox{ :
}n\in\mathbb{N}\}\nonumber
    \end{eqnarray}
\end{enumerate}
\end{proposition}

\textbf{Proof of Theorem 7. }The mapping $\textbf{T}$ may be
represented through the diagram\\
\\
\\
\begin{math}
\setlength{\unitlength}{1cm} \thicklines
\begin{picture}(13,6)

\put(3.4,5.2){$\mathcal{ML}^{m}\left(\Omega\right)^{K}$}
\put(5.1,5.3){\vector(1,0){6.4}}
\put(11.6,5.2){$\mathcal{ML}^{0}\left(\Omega\right)^{K}$}
\put(7.2,0.9){$\mathcal{ML}^{0}\left(\Omega\right)^{M}$}
\put(7.8,5.7){$\textbf{T}$} \put(4.0,5.0){\vector(1,-1){3.5}}
\put(5.0,3.2){$\textbf{D}$}
\put(10.7,3.2){$\overline{\textbf{F}}$}
\put(8.5,1.5){\vector(1,1){3.5}}

\end{picture}
\end{math}\\
\\
where $\textbf{D}$ maps $\textbf{u}$ to its vector of derivatives,
that is,
\begin{eqnarray}
\textbf{D}:\mathcal{ML}^{m}\left(\Omega\right)^{K}\ni
\textbf{u}\mapsto
\left(\mathcal{D}^{\alpha}u_{i}\right)_{|\alpha|\leq m,i\leq K}
\in\mathcal{ML}^{0}\left(\Omega\right)^{M}\label{DDef}
\end{eqnarray}
and $\textbf{F}=\left(\overline{F}_{i}\right)_{i\leq K}$ is
defined componentwise through
\begin{eqnarray}
\overline{F}_{i}:\mathcal{ML}^{0}\left(\Omega\right)^{M}\ni\textbf{u}
\mapsto \left(I\circ
S\right)\left(F_{i}\left(\cdot,u_{1},...,u_{M}\right)\right)
\in\mathcal{ML}^{0}\left(\Omega\right) \label{FDef}
\end{eqnarray}
Clearly the mapping $\textbf{D}$ is uniformly continuous, so in
view of the diagram it suffices to show that
$\overline{\textbf{F}}$ is uniformly continuous with respect to
the product uniform
convergence structure.\\
In this regard, we consider sequences of order intervals
$\left(I_{n}^{i}\right)$ in $\mathcal{ML}^{0}\left(\Omega\right)$,
where $i=1,...,M$, that satisfies condition 2) of Definition
\ref{JoDef}.  We claim
\begin{eqnarray}
\begin{array}{ll}
\forall & n\in\mathbb{N}\mbox{ :} \\
\exists & J_{n}^{1},...,J_{n}^{K}\subseteq \mathcal{ML}^{0}\left(\Omega\right)\mbox{ order intervals :} \\
& F_{j}\left(\prod_{i=1}^{M}I_{n}^{i}\right)\subseteq J_{n}^{j}\mbox{, }j=1,...,K \\
\end{array}\label{FClaim}
\end{eqnarray}
To verify (\ref{FClaim}), observe that there is a closed nowhere
dense set $\Gamma_{n}\subseteq \Omega$ so that
\begin{eqnarray}
\begin{array}{ll}
\forall & x\in\Omega\setminus\Gamma\mbox{ :} \\
\exists & a\left(x\right)>0\mbox{ :} \\
\forall & i=1,...,M\mbox{ :} \\
& u\in I_{n}^{i}\Rightarrow |u\left(x\right)|\leq a\left(x\right) \\
\end{array}\label{PB1}
\end{eqnarray}
Since $F_{j}:\Omega\times\mathbb{R}^{M}\rightarrow \mathbb{R}$ is
continuous, it follows from (\ref{PB1}) that
\begin{eqnarray}
\begin{array}{ll}
\forall & x\in\Omega\setminus\Gamma\mbox{ :} \\
\exists & b\left(x\right)>0\mbox{ :} \\
& \left(\begin{array}{ll}
\forall & i=1,...,M\mbox{ :} \\
& u_{i}\in I_{n}^{i} \\
\end{array}\right)\Rightarrow |F_{j}\left(x,u_{1}\left(x\right),...,u_{M}\left(x\right)\right)|\leq b\left(x\right) \\
\end{array}\label{PB2}
\end{eqnarray}
Therefore, in view of Proposition \ref{NLBounded}, our claim
(\ref{FClaim}) holds.  In particular, since
$\mathcal{NL}\left(\Omega\right)$ is Dedekind complete
\cite{vdWalt3}, we may set
\begin{eqnarray}
J_{n}^{j}=[\lambda_{n}^{j},\mu_{n}^{j}]\nonumber
\end{eqnarray}
where, for each $n\in\mathbb{N}$ and each $j=1,...,K$
\begin{eqnarray}
\lambda_{n}^{j}=\inf\{\overline{F}_{j}\textbf{u}\mbox{ :
}\textbf{u}\in \prod_{i=1}^{M}I_{n}^{i}\}\nonumber
\end{eqnarray}
and
\begin{eqnarray}
\mu_{n}^{j}=\sup\{\overline{F}_{j}\textbf{u}\mbox{ :
}\textbf{u}\in \prod_{i=1}^{M}I_{n}^{i}\}\nonumber
\end{eqnarray}
The sequence $\left(\lambda_{n}^{j}\right)$ and
$\left(\mu_{n}^{j}\right)$ are increasing and decreasing,
respectively.  For each $j=1,...,K$ we may consider
\begin{eqnarray}
\sup\{\lambda_{n}^{j}\mbox{ : }n\in\mathbb{N}\}=u^{j} \leq
v^{j}=\inf\{\mu_{n}^{j}\mbox{ : }n\in\mathbb{N}\}\nonumber
\end{eqnarray}
We claim that $u^{j}=v^{j}$.  To see this, we note that for each
$j=1,...,K$ there is some
$w^{j}\in\mathcal{NL}\left(\Omega\right)$ so that
\begin{eqnarray}
\sup\{l_{n}^{j}\mbox{ : }n\in\mathbb{N}\}
=w^{j}=\inf\{u^{j}_{n}\mbox{ : }n\in\mathbb{N}\}\nonumber
\end{eqnarray}
where $I_{n}^{j}=[l_{n}^{j},u^{j}_{n}]$.  Applying Lemma
\ref{IncPW} and the continuity of $F_{j}$ our claim is verified.
Applying Propositions \ref{MLDense} and \ref{NLBounded} we obtain
sequence $\left(\overline{I}_{n}^{j}\right)$ of order intervals in
$\mathcal{ML}^{0}\left(\Omega\right)$ that satisfies condition 2)
of Definition \ref{JoDef} and
\begin{eqnarray}
\overline{F}_{j}\left(\prod_{i=1}^{M}I_{n}^{i}\right) \subseteq
\overline{I}_{n}^{j}\nonumber
\end{eqnarray}
This completes the proof.\mbox{ } $\blacksquare$\\ \\
The proof also relies on the following lemma.
\begin{lemma}\label{IncPW}
Consider a decreasing sequence $\left(u\right)$ in
$\mathcal{NL}\left(\Omega\right)$ that satisfies
\begin{eqnarray}
u=\inf\{u_{n}\mbox{ : }n\in\mathbb{N}\}\in\mathcal{NL}
\left(\Omega\right)\nonumber
\end{eqnarray}
Then the follows holds:
\begin{eqnarray}
\begin{array}{ll}
\forall & \epsilon>0\mbox{ :} \\
\exists & \Gamma_{\epsilon}\subseteq\Omega\mbox{ closed nowhere dense :} \\
& x\in\Omega\setminus\Gamma_{\epsilon}\Rightarrow \left(\begin{array}{ll}
\exists & N_{\epsilon}\in\mathbb{N}\mbox{ :} \\
& u_{n}\left(x\right)-u\left(x\right)<\epsilon\mbox{, }n\geq N_{\epsilon} \\
\end{array}\right) \\
\end{array}\nonumber
\end{eqnarray}
The corresponding statement for increasing sequences is also true.
\end{lemma}
\begin{proof}
Take $\epsilon>0$ arbitrary but fixed.  We start with the set
\begin{eqnarray}
C=\left\{x\in\Omega\begin{array}{|ll}
\forall & n\in\mathbb{N}\mbox{ :} \\
& u_{n}\mbox{, }u\mbox{ continuous at }x \\
\end{array}\right\}\nonumber
\end{eqnarray}
which must have complement a set of First Baire Category, and
hence it is dense.  In view of (\ref{IneqDSet}), the set of points
\begin{eqnarray}
C_{\epsilon}=\left\{x\in C\begin{array}{|ll}
\exists & N_{\epsilon}\in\mathbb{N}\mbox{ :} \\
& u_{n}\left(x\right)-u\left(x\right)<\epsilon\mbox{, }n\geq N_{\epsilon} \\
\end{array}\right\}\nonumber
\end{eqnarray}
must be dense in $\Omega$.  From the continuity of $u$ and the
$u_{n}$ on $C$ it follows that
\begin{eqnarray}
\begin{array}{ll}
\forall & x_{0}\in C_{\epsilon}\mbox{ :}  \\
\exists & \delta_{x_{0}}>0\mbox{ :} \\
& x\in C\mbox{, }\|x-x_{0}\|<\delta\Rightarrow x\in C_{\epsilon} \\
\end{array}\nonumber
\end{eqnarray}
Since $C$ is dense in $\Omega$, the result follows.
\end{proof}

\end{article}
\end{document}